\documentclass[12pt,leqno]{amsart}
\usepackage{amsmath,verbatim,amssymb,amsfonts,latexsym,amsmath,amsthm,tikz,enumerate}

\newtheorem{theorem}{Theorem}[section]
\newtheorem{lemma}[theorem]{Lemma}
\newtheorem{corollary}[theorem]{Corollary}
\newtheorem{proposition}[theorem]{Proposition}
\newtheorem{question}[theorem]{Question}
\theoremstyle{definition}

\theoremstyle{remark}

\newtheorem*{claim}{Claim}
\newtheorem*{sub-claim}{sub-claim}

\newcommand{\explicitSet}[1]{\left\lbrace #1 \right\rbrace}
\newcommand{\brackets}[1]{\left\langle #1 \right\rangle}
\newcommand{\set}[2]{\explicitSet{#1 \colon #2}}
\newcommand{\seq}[2]{\brackets{#1 \colon #2}}

\newcommand{\0}{\emptyset}
\renewcommand{\a}{\alpha}
\renewcommand{\b}{\beta}

\newcommand{\card}[1]{\left\lvert #1 \right\rvert}

\newcommand{\continuum}{\mathfrak{c}}

\newcommand{\p}{\mathbb{P}}

\newcommand{\w}{\omega}
\renewcommand{\k}{\kappa}
\newcommand{\sub}{\subseteq}
\newcommand{\rest}{\!\restriction\!}

\newcommand{\homeo}{\cong}

\newcommand{\closure}[1]{\overline{#1}}

\newcommand{\s}{\sigma}

\newcommand{\N}{\w^\w}
\newcommand{\C}{2^\w}
\renewcommand{\l}[1]{\mathrm{lev\!\left(#1\right)}}

\newcommand{\cat}{\,\!^\frown}

\newcommand{\lev}[2]{\mathrm{Lev}_{#1}\!\left(#2\right)}
\renewcommand{\[}{[\![}
\renewcommand{\]}{]\!]}

\def\al{\alpha}
\def\ff{\mathcal{F}}
\def\fin{\mbox{FIN}}
\def\ga{\gamma}
\def\ka{\kappa}
\def\om{\omega}
\def\sm{\setminus}
\def\su{\subseteq}

\begin{document}

\title{Partitions of $2^\w$ and Completely Ultrametrizable Spaces}

\author[W. R. Brian]{William R. Brian}
\address {
William R. Brian\\
Department of Mathematics\\
Tulane University\\
6823 St. Charles Ave.\\
New Orleans, LA 70118}
\email{wbrian.math@gmail.com}
\urladdr{wrbrian.wordpress.com}

\author[A. W. Miller]{Arnold W. Miller}
\address {
Arnold W. Miller \\
Department of Mathematics\\
University of Wis\-con\-sin-Madison \\
480 Lincoln Drive \\
Madison, WI 53706-1388}
\email{miller@math.wisc.edu}
\urladdr{www.math.wisc.edu/$\sim$miller}

\keywords{complete ultrametric, tree, condensation, partition, Borel set, Cantor space}

\begin{abstract}
We prove that, for every $n$, the topological space $\w_n^\w$ (where $\w_n$ has the discrete topology) can be partitioned into $\aleph_n$ copies of the Baire space. Using this fact, the authors then prove two new theorems about completely ultrametrizable spaces. We say that $Y$ is a condensation
of $X$ if there is a continuous bijection $f:X\to Y$.  First, it is proved that $\w^\w$ is a condensation of $\w_n^\w$ if and only if $\w^\w$ can be partitioned into $\aleph_n$ Borel sets, and some consistency results are given regarding such partitions. It is also proved that it is consistent with \textit{ZFC} that, for any $n < \w$, $\continuum = \w_n$ and there are exactly $n+3$ similarity types of perfect completely ultrametrizable spaces of size $\continuum$. These results answer two questions of the first author from \cite{Brn}.
\end{abstract}

\maketitle

\section{Introduction}

Every zero-dimensional Polish space can be represented as the end space of a countable tree. This fact is exploited in many classical proofs: that every perfect Polish space is a condensation of $\N$, the Alexandrov-Urysohn characterization of $\N$, Brouwer's characterization of $\C$. All these can be proved by combining a bit of topology with a bit of clever reasoning with trees.

In \cite{Brn}, the first author explores to what extent these classical proofs involving countable trees can be translated to the context of uncountable trees. The end spaces of arbitrary trees are precisely the completely ultrametrizable spaces, and several classical results for Polish spaces extend to this class. However, many proofs break down in the uncountable case. In general, it seems that, while many strong results for countable trees are provable in \textit{ZFC}, the corresponding results for uncountable trees often require extra set-theoretic hypotheses.

One recurring theme in the analysis of uncountable trees is that of partitioning the spaces $\C$ and $\N$. For example, it is proved in \cite{Brn} (Theorem 6.2) that $\N$ is a condensation of $\k^\w$ whenever $\N$ can be partitioned into $\k$ Borel sets. It is also proved (Theorem 5.1) that $\C$ is a condensation (that is, a continuous bijective image) of every perfect completely ultrametrizable space of size $\card{\C}$ if and only if $\C$ can be partitioned into $\k$ copies of $\C$ for every $\k \leq \continuum$.

Theorem~\ref{thm:inductivedivisions} below, which is our main lemma for what follows, can be viewed as another partition theorem for completely ultrametrizable spaces. It states that $\w_{\a+1}^\w$ can be partitioned into $\aleph_{\a+1}$ copies of $\w_\a^\w$. For $\k$ below the first singular cardinal, a simple induction then allows us to write $\k^\w$ as a union of $\k$ copies of the Baire space $\N$.

The second author, in \cite{AM1}, develops methods for finding models of \textit{ZFC} in which $\N$ can be partitioned in nice ways. Here we refine these techniques (see Theorem~\ref{thm:lotsofpartitions}) and, using the partition theorem of the previous paragraph, obtain nice partitions of the spaces $\k^\w$ for $\k < \aleph_\w$. This allows us to obtain several new consistency results for completely ultrametrizable spaces. These consistency results will be the topic of Section~\ref{sec:partitions}.

In Section~\ref{sec:types}, we present a more complex application of Theorem~\ref{thm:lotsofpartitions} to answer an open question from \cite{Brn}. Specifically, we show that, for any value of $\continuum < \w_\w$, it is consistent to have the minimum possible number of similarity classes of perfect completely ultrametrizable spaces of size $\continuum$ ($X$ and $Y$ are similar if each is a condensation of the other).

\section{Completely ultrametrizable spaces and trees}\label{sec:trees}

In what follows, if an ordinal or cardinal is treated as a topological space then it is assumed to have the discrete topology. As usual, every ordinal is equal to the set of its predecessors. Most of our terminology is standard, and the rest follows \cite{Brn}. In the interest of making this paper self-contained, we will review in this section some definitions and preliminaries from \cite{Brn} concerning trees and completely ultrametrizable spaces.

A \textbf{tree} is a connected, nonempty, infinite graph in which every two nodes are connected by exactly one path, together with a distinguished node called the \textbf{root}. If $T$ is a tree and $s,t \in T$, we say that $t$ \textbf{extends} $s$ if the unique path from the root to $t$ goes through $s$. We denote the extension relation by $\leq$.

The extension relation allows us to think of trees as partial orders, and in what follows we will freely confuse trees with ``tree-like'' partial orders. Especially ubiquitous is $\k^{<\w}$, the set of all finite sequences in $\k$, ordered by inclusion.

Two nodes of a tree $T$ are \textbf{incomparable} if neither one extends the other. A tree is \textbf{pruned} if each of its elements has a proper extension, and is \textbf{perfect} if each of its elements has two incomparable proper extensions. In what follows, a ``tree'' will always mean a pruned, nonempty tree.

If $T$ is a tree and $s$ is a node of $T$, then
$$T_s = \set{t \in T}{t \leq s \text{ or } s \leq t}$$
is the set of all nodes of $T$ that compare with $s$ under the extension relation. If $s \lneq t$ and there is no $r$ such that $s \lneq r \lneq t$, then $t$ is a \textbf{child} of $s$.

In every tree $T$, there is a unique path from the root to a given node. This naturally divides $T$ into levels. We say that a node $s$ is at \textbf{level} $n$, denoted $\l{s} = n$, if the unique path from the root to $s$ has $n-1$ elements. Thus the root is the unique node at level $0$, the children of the root are all at level $1$, etc. We write $\lev{n}{T}$ for $\set{s \in T}{\l{s} = n}$.

A \textbf{branch} of a tree $T$ is an infinite sequence $x$ of nodes in $T$ such that $x(0)$ is the root and $x(n+1)$ is a child of $x(n)$ for every $n$. $\[T\]$ is the set of all branches of $T$. $\[T\]$ has a natural topology defined by taking $\set{\[T_s\]}{s \in T}$ to be a basis. This space is sometimes called the \textbf{end space} of $T$.

\begin{proposition}\label{prop:whatarethey}
If $T$ is a (perfect) tree, then $\[T\]$ is a (perfect) completely ultrametrizable space. If $X$ is a (perfect) completely ultrametrizable space, then there is a (perfect) tree $T$ such that $\[T\] \homeo X$.
\end{proposition}
\begin{proof}
See \cite{Huh} for a thorough treatment.
\end{proof}

We will say that a space $X$ is \textbf{represented} by a tree $T$ whenever $X \homeo \[T\]$. Thus Proposition~\ref{prop:whatarethey} can be rephrased by saying that the (perfect) completely ultrametrizable spaces are precisely those representable by (perfect) trees. Note that, for all cardinals $\k$, $\k^{<\w}$ represents $\k^\w$. In particular, $2^{<\w}$ represents the Cantor space and $\w^{<\w}$ represents the Baire space.

Let $T$ be a tree and let $X$ be a topological space. A $T$\textbf{-scheme} on $X$ is a family $(B_s)_{s \in T}$ of subsets of $X$ such that
\begin{itemize}
\item $B_t \sub B_s$ whenever $t$ is an extension of $s$.
\item $B_s \cap B_t = \0$ whenever $s$ and $t$ are incompatible.
\end{itemize}
If $d$ is a metric on $X$ then $(B_s)_{s \in T}$ has \textbf{vanishing diameter} (with respect to $d$) if $\lim_{n \rightarrow \infty}\mathrm{diam}(B_{x(n)}) = 0$ whenever $x \in \[T\]$. If $X$ is a metric space and $(B_s)_{s \in T}$ is a $T$-scheme with vanishing diameter, then let $D = \set{x \in \[T\]}{\bigcap_{n \in \w}A_{x(n)} \neq \0}$ and define $f: D \rightarrow X$ by $\{f(x)\} = \bigcap_{n < \w}B_{x(n)}$. We call $f$ the \textbf{associated map}.

\begin{lemma}\label{lem:schemes}
Let $(B_s)_{s \in T}$ be a $T$-scheme with vanishing diameter on a metric space $(X,d)$. If $f: D \rightarrow X$ is the associated map, then
\begin{enumerate}
\item $f$ is injective and continuous.
\item if $B_s = \bigcup \set{B_t}{t \text{ is a child of }s}$ for all $s \in T$, then $f$ is surjective.
\item if $B_s$ is open for every $s \in T$, then $f$ is open.
\end{enumerate}
\end{lemma}
\begin{proof}
See \cite{Brn}, Lemma 2.3.
\end{proof}

\section{Partitions of large spaces}\label{sec:partitions}

As mentioned in the introduction, the existence of well-behaved partitions of completely ultrametrizable spaces determines a good deal about the maps between them. In this section we prove our main lemma concerning partitions of the spaces $\w_n^\w$, and use this lemma to derive some consistency results.

\begin{lemma}\label{lem:disjointunions}
Let $T$ be any tree. Then every open subset of $\[T\]$ can be written as a disjoint union of sets of the form $\[T_s\]$.
\end{lemma}
\begin{proof}
Let $U$ be any open subset of $\[T\]$. For each $s \in \[T\]$, say $s \in A$ if and only if $\[T_s\] \sub U$ and $\[T_t\] \not\sub U$ for any $t$ such that $t \lneq s$. No two elements of $A$ are comparable, and if $x \in U$ then $x(n) \in A$ for some $n$, namely the smallest $n$ such that $\[T_{x(n)}\] \sub U$. Therefore $U = \bigcup_{s \in A}\[T_s\]$, and this is a disjoint union.
\end{proof}

\begin{lemma}\label{lem:opensubsets}
Let $\k$ be an infinite cardinal. Every nonempty open subset of $\k^\w$ is homeomorphic to $\k^\w$.
\end{lemma}
\begin{proof}
Let $U$ be an open subset of $\k^\w$. Since $\k^\w \homeo \[\k^{<\w}\]$, we may apply Lemma~\ref{lem:disjointunions} to obtain $U$ as a disjoint union of sets of the form $\[\k^{<\w}_s\]$. Each such set is clearly a copy of $\k^\w$, and we have at most $\k$ of them since $\card{\k^{<\w}} = \k$. Thus, for some $\lambda \leq \k$, $U \homeo \lambda \times \k^\w \homeo \k^\w$.
\end{proof}

\begin{lemma}\label{lem:thebreakup}
Let $\k$ be an infinite cardinal and $d$ any metric on $\k^\w$ compatible with its topology. If $\varepsilon > 0$, then there is a partition $\set{B_\a}{\a \in \k}$ of $\k^\w$ into $\k$ clopen sets such that $\mathrm{diam}(B_\a) < \varepsilon$ for every $\a$.
\end{lemma}
\begin{proof}
See \cite{Brn}, Lemma 6.7.
\end{proof}

\begin{lemma}\label{lem:denseGdelta}
Let $\k$ be an infinite cardinal. If $X$ is a dense $G_\delta$ subset of $\k^\w$, then $X \homeo \k^\w$.
\end{lemma}
\begin{proof}
Recall (e.g., from \cite{AJL}, Corollary 5) that a metric space is completely ultrametrizable if and only if it is zero-dimensional and \v{C}ech complete. Since metrizability, zero-dimensionality, and \v{C}ech completeness are all inherited by $G_\delta$ subsets, $X$ is completely ultrametrizable.

Suppose that $U$ is any clopen subset of $X$. There is some $s \in \k^{<\w}$ such that $\[\k^{<\w}_s\] \cap X \sub U$. Because $X$ is dense in $\k^\w$, $\[\k^{<\w}_{s \cat \a}\] \cap X$ is a nonempty clopen subset of $U$ for every $\a < \k$. Consequently, $\{X \setminus \[\k^{<\w}_s\]\} \cup \set{\[\k^{<\w}_{s \cat \a}\] \cap X}{\a \in \k}$ is a partition of $U$ into clopen sets. Thus every open neighborhood of $U$ can be partitioned into $\k$ disjoint clopen subsets.

Fix a compatible complete metric on $X$. We will show that $X \homeo \k^\w$ by constructing an appropriate $\k^{<\w}$-scheme in $X$.

Let $B_\0 = X$ and fix $s \in \k^{<\w}$. Suppose $B_s$ has been defined, is a clopen subset of $X$, and has diameter at most $\frac{1}{\l{s}+1}$. As we have already noted, it is possible to partition $B_s$ into $\k$ disjoint clopen subsets. If necessary, we may use Lemma~\ref{lem:thebreakup} to partition each of these further into sets smaller than $\frac{1}{\l{s}+2}$; thus we may assume that $B_s$ is partitioned into $\k$ clopen sets, each smaller than $\frac{1}{\l{s}+2}$. Enumerate these as $\set{B_{s \cat \a}}{\a \in \k}$ to define $B_t$ for every child $t$ of $s$. By recursion, this defines a $\k^{<\w}$-scheme $(B_s)_{s \in \k^{<\w}}$ in $X$.

By construction, this scheme has vanishing diameter, $B_\0 = X$, each $B_s$ is clopen, and $B_s = \bigcup \set{B_t}{t \text{ is a child of } s}$ for every $s$. By Lemma \ref{lem:schemes}, the associated map of $(B_s)_{s \in \k^{<\w}}$ is a homeomorphism. It remains to show that the domain of this map is $X$: that is, we must show that $\bigcap_{n \in \w}B_{x(n)} \neq \0$ for every $x \in \[\k^{<\w}\]$. Since each $B_s$ is clopen, we have $\bigcap_{n \in \w}B_{x(n)} = \bigcap_{n \in \w}\closure{B_{x(n)}}$, and this is nonempty because the $B_{x(n)}$ have decreasing diameter with respect to a complete metric.
\end{proof}

Note that, if we set $\k = \w$, then Lemma~\ref{lem:denseGdelta} reduces to a variant of a classical theorem of Alexandrov and Urysohn (see \cite{Kec}, Theorem 7.7).

\begin{theorem}\label{thm:inductivedivisions}
Let $\a$ be any ordinal. Then $\w_{\a+1}^\w$ can be partitioned into $\aleph_{\a+1}$ homeomorphic copies of $\w_\a^\w$.
\end{theorem}
\begin{proof}
Fix an ordinal $\a$. For every $\b$ with $\w_\a \leq \b < \w_{\a+1}$, let
$$X_\b = \b^\w \setminus \bigcup \set{\gamma^\w}{\w_\a \leq \gamma < \b}.$$
Because $\w_{\a+1}$ is a regular uncountable cardinal, the range of every $x \in \w_{\a+1}^\w$ is bounded by some $\b < \w_{\a+1}$, so that $x \in \b^\w$. Therefore $\set{X_\b}{\w_\a \leq \b < \w_{\a+1}}$ is a partition of $\w_{\a+1}^\w$. Of course $X_{\w_\a} = \w_\a^\w$. For $\b \neq \w_\a$, we will show that $X_\b$ is empty if and only if $\b$ has uncountable cofinality, and otherwise is homeomorphic to $\w_\a^\w$.

First suppose that $\b$ has uncountable cofinality. Then every element of $\b^\w$ is bounded inside some $\gamma^\w$, $\gamma < \b$, so that $\b^\w = \bigcup_{\gamma < \b}\gamma^\w$. If $\b \neq \w_\a$, then this implies $X_\b = \0$.

Next suppose that $\b = \gamma+1$ is a successor ordinal. Because $\gamma$ is closed in $\b$, $\gamma^\w$ is closed in $\b^\w$. Also, $\b^\w \setminus \gamma^\w \neq \0$ (explicitly, $\b^\w \setminus \gamma^\w$ is the set of all sequences in $\b$ with the point $\gamma$ in their range). By Lemma~\ref{lem:opensubsets}, $X_\b \homeo \w_\a^\w$.

Finally, suppose that $\b$ is a limit ordinal with countable cofinality. Let $\seq{\gamma_n}{n < \w}$ be a sequence of ordinals with limit $\b$. As in the previous paragraph, each $\gamma_n^\w$ is closed in $\b^\w$. Also, each $\gamma_n^\w$ is nowhere dense in $\b^\w$. To see this, let $s \in \b^{<\w}$ so that $U = \[\b^{<\w}_s\]$ is a basic open set of $\b^\w$, and consider that $\[\b^{<\w}_{s \cat \gamma_n}\]$ is an open subset of $U$ that is disjoint from $\gamma_n^\w$. Since $X_\b = \b^\w \setminus \bigcup_{n < \w}\gamma_n^\w$, $X$ is $G_\delta$ in $\b^\w$, and $X$ is dense in $\b^\w$ by the Baire Category Theorem. Since $\b^\w \homeo \w_\a^\w$, it follows from Lemma~\ref{lem:denseGdelta} that $X_\b \homeo \w_\a^\w$.
\end{proof}

Theorem~\ref{thm:inductivedivisions} can be seen as a topological version of the basic fact of cardinal arithmetic that $\aleph_{\a+1}^{\aleph_0} = \aleph_{\a+1} \cdot \aleph_\a^{\aleph_0}$.

\begin{corollary}\label{cor:inductivedivisions}
Let $n < \w$. Then $\w_n^\w$ can be partitioned into $\aleph_n$ copies of $\N$.
\end{corollary}
\begin{proof}
By induction, using Theorem~\ref{thm:inductivedivisions}
\end{proof}

Corollary~\ref{cor:inductivedivisions} cannot be extended to $\k \geq \aleph_\w$ using \textit{ZFC} alone, because the induction breaks down at the first singular cardinal. In fact, it always fails when $\k$ has countable cofinality:

\begin{theorem}
If $\ka$ is an uncountable cardinal with cofinality $\om$,  then
$\ka^\om$ is not the union of $\ka$ many subspaces homeomorphic
to $\N$.
\end{theorem}

\begin{proof}
If $X\su \ka^\om$ is homeomorphic to $\N$,
then $X\su \Gamma^\om$ for some countable $\Gamma\su \ka$.
To see this, note that for any $n<\om$ if
$\pi_n:\ka^\om\to \ka$ is the projection map $\pi(x)=x(n)$,
then $\pi_n(X)$ must be countable since otherwise $X$ would contain
an uncountable family of pairwise disjoint open sets.

Suppose $X_\al\su \ka^\om$ for $\al<\ka$ are homeomorphic 
copies of $\N$.  Let
$\Gamma_\al\su\ka$ be countable with $X_\al\su\Gamma_\al^\om$.
Let $\ka_n$ be a cofinal sequence in $\ka$ and choose
$$x(n)\in \ka\sm\bigcup_{\al<\ka_n}\Gamma_\al$$
for each $n$.
Then $x\in \ka^\om \sm \bigcup_{\al<\ka} X_\al$.
\end{proof}

Note that
for any $\ka$ the space $\ka^\om$ can be partitioned into
$|\ka^\om|$ copies of $\N$.  This is because $\om^\om\times\ka^\om$
is homeomorphic to $\ka^\om$.  
Hence, assuming GCH, for every
$\ka$ with uncountable cofinality the space $\ka^\om$ 
can be partitioned into $\ka$ many copies of $\N$.
This would also follow from
the weaker assumption that $|\ka^\om|=\ka$ for every
$\ka>\aleph_\om$ with uncountable cofinality.

\begin{question}
Suppose $\aleph_\om<\ka<\continuum$ and $\ka$ has uncountable
cofinality,  then can
$\ka^\om$ be partitioned into $\ka$ many copies of $\N$?
\end{question}

In \cite{Brn} it is proved that $\N$ is a condensation of $\k^\w$ whenever $\N$ can be partitioned into $\k$ Borel sets. Theorem~\ref{thm:inductivedivisions} allows us to prove a partial converse.

\begin{theorem}\label{thm:unexpectedconverse}
Let $n < \w$. Then $\N$ is a condensation of $\w_n^\w$ if and only if $\N$ can be partitioned into $\aleph_n$ Borel sets.
\end{theorem}
\begin{proof}
Theorem 6.2 of \cite{Brn} states that $\N$ is a condensation of $\k^\w$ whenever $\N$ can be partitioned into $\k$ Borel sets. For the converse, suppose that $f: \w_n^\w \to \N$ is a condensation. By Corollary~\ref{cor:inductivedivisions}, there is a partition $\set{N_\a}{\a \in \w_n}$ of $\w_n^\w$ into $\aleph_n$ copies of $\N$. By an old theorem of Lusin and Souslin, every bijective continuous image of $\N$ is Borel (see \cite{Kec}, Theorem 15.1). Thus $\set{f(N_\a)}{\a \in \w_n}$ is a partition of $\N$ into $\aleph_n$ Borel sets.
\end{proof}

The following corollary answers Question 6.4 from \cite{Brn}.

\begin{corollary}
It is relatively consistent with \textit{ZFC} that
$\continuum=\om_3$ and there is no condensation $\om_2^\w \to \N$.
\end{corollary}
\begin{proof}
In \cite{AM2}, Theorem 3.7, the second author proves that if $\w_3$ Cohen reals are added to a model of \textit{CH} then there is no partition of $\C$ into $\aleph_2$ Borel sets. Since $\N$ can be identified with a co-countable subset of $\C$, this model also has no partition of $\N$ into $\aleph_2$ Borel sets. It follows from Theorem~\ref{thm:unexpectedconverse} that this model has no condensation $\w_2^\w \to \N$.
\end{proof}

Contrast this result with the result of Hausdorff in \cite{Hsd}, where it is proved from \textit{ZFC} that $\N$ can be partitioned into $\aleph_1$ Borel sets (and hence $\N$ is a condensation of $\w_1^\w$). The next theorem gives the opposite consistency result:

\begin{theorem}\label{thm:lotsofpartitions}
It is consistent with any possible value of $\continuum$ that for every $\k \leq \continuum$ there is a partition of $\C$ into $\k$ closed sets.
\end{theorem}
\begin{proof}
It is proved in \cite{AM1}, Theorem 4, that for any possible value of $\continuum$ and any fixed $\k < \continuum$, there is a model in which $\C$ can be partitioned into $\k$ copies of $\C$. Here we show how to modify that construction to obtain a partition into $\k$ copies of $\C$ for all $\k < \continuum$ simultaneously.

If \textit{CH} holds then the conclusion is trivial. Let $M$ be any model in which \textit{CH} fails. We will show how to construct a finite support iterated forcing such that extending $M$ by this forcing preserves cardinals, does not change the value of $\continuum$, and adds for every $\w_1 \leq \k < \continuum$ a partition of $\C$ into $\k$ closed sets.

For $X \sub 2^\w$, define $\p(X)$ as follows. Conditions are finite mutually consistent sets of sentences of the form ``$\[2^{<\w}_s\] \cap C_n = \0$'' or ``$x \in C_n$'' where $n \in \w$, $x \in X$, and $s \in 2^{<\w}$ (this is as in \cite{AM1}). In an extension of $M$ by $\p(X)$, $\bigcup_{n \in \w}C_n$ will be an $F_\s$ set that covers $X$ and misses every element of $(\C)^M \setminus X$. For any $X$, $\p(X)$ has the c.c.c.

Our iterated forcing will have length $\a \cdot \w_1$, where $\a$ is the unique ordinal such that $\continuum = \w_\a$ (alternatively, this can be thought of as a length-$\w_1$ iteration of length-$\a$ iterations). For each $\k$ with $\w_1 \leq \k < \continuum$, let $X_\k^0$ be a subset of $\C$ with $\card{X_\k^0} = \k$. Let $\delta < \a \cdot \w_1$, and let $\b$ and $\gamma$ be the unique ordinals such that $\gamma < \a$ and $\delta = \a \cdot \b + \gamma$ (see \cite{Kun}, Ch. I, Ex. 3 to see that such ordinals exist and are unique). In our iteration, $M_{\delta+1}$ is obtained by forcing with $\p(\C \setminus (X_{\w_\gamma} \cup \bigcup \{ F^\xi_{\w_\gamma} : \xi < \b \}))$ in $M_\delta$; this creates a generic $F_\s$ set which we call $F^\b_{\w_\gamma}$.

Since each $\p(X)$ has the c.c.c., our iteration has the c.c.c., and since $\card{\a \cdot \w_1} \leq \w_\a = \continuum^M$, $M_{\a \cdot \w_1} \models `` \continuum = \w_\a "$. Fix $\gamma < \a$ and let $\k = \w_\gamma < \continuum$. For every $x \in \C \setminus X_\k$, there is smallest $\b < \w_1$ such that $x \in M_{\a \cdot \b +\gamma}$, in which case $x \in F^\b_\k$. Thus $M_{\a \cdot \w_1} \models `` \C = X_\k \cup \bigcup_{\b < \w_1} F^\b_\k "$. In other words, $\set{\{x\}}{x \in X_{\k}} \cup \set{F^\b_\k}{\b < \w_1}$ is a partition of $\C$ in $M_{\a \cdot \w_1}$, and (because our iteration preserves cardinals) this partition has cardinality $\k$ when $\k$ is uncountable.

To obtain a partition of $\C$ into $\k$ closed sets in $M_{\a \cdot \w_1}$, it is sufficient to note that every $F_\s$ subset of $\C$ can be partitioned into countably many compact sets. This is showed in \cite{AM1} (the last part of the proof of Theorem 4) or, alternatively, in \cite{Brn} (Proposition 3.5).
\end{proof}

\begin{corollary}
It is consistent with any possible value of $\continuum$ that whenever $\w \leq \k \leq \continuum$ there is a condensation $\k^\w \to \N$.
\end{corollary}
\begin{proof}
Because $\N$ can be identified with a co-countable subset of $\C$, the model in Theorem~\ref{thm:lotsofpartitions} has, for every $\k < \continuum$, a partition of $\N$ into $\k$ Borel sets. The corollary now follows from Theorem~\ref{thm:unexpectedconverse}.
\end{proof}

One might notice that the $\k$-sized partition given by our forcing consists of $\k$ singletons and at most $\aleph_1$ nontrivial closed sets. However, it is easy to modify these partitions to obtain $\k$ copies of $\C$:

\begin{proposition}\label{prop:notjustpoints}
If $\C$ can be partitioned into $\k$ closed sets, then $\C$ can be partitioned into $\k$ copies of $\C$.
\end{proposition}
\begin{proof}
Let $\set{K_\a}{\a \in \k}$ be a partition of $\C$ into closed sets. Then $\set{K_\a \times \C}{\a \in \k}$ is a partition of $\C \times \C \homeo \C$ into copies of $\C$.
\end{proof}

\begin{corollary}
It is consistent with any possible value of $\continuum$ that $\C$ is a condensation of every perfect completely ultrametrizable space $X$ with $\card{X} = \continuum$.
\end{corollary}
\begin{proof}
Combine Theorem~\ref{thm:lotsofpartitions}, Proposition~\ref{prop:notjustpoints}, and Theorem 5.1 of \cite{Brn}.
\end{proof}

A result of Hausdorff states that there is always a partition of $\N$ into $\aleph_1$ Borel sets (see \cite{Hsd}). Theorem 3.7 of \cite{AM2}, together with the comments that follow it, says that it is consistent with any value of $\continuum$ that every uncountable partition of $\N$ into Borel sets has size $\aleph_1$ or $\continuum$ (this is strengthened below in Corollary~\ref{cohenreal}). Theorem~\ref{thm:lotsofpartitions} gives the opposite result, but it remains an open problem to find models in which some intermediate property holds. Is it consistent, for example, to have $\continuum = \w_9$ with uncountable partitions of sizes $\aleph_1$, $\aleph_4$, $\aleph_8$, and $\aleph_9$, but of no other sizes? The following propositions provide partial answers to such questions.

Let $\fin(\ka,2)$ be the partial order of finite partial functions
from $\ka$ to $2$, i.e., Cohen forcing.

\begin{theorem}
(\cite{AM2} 3.7) 
Suppose $M$ is a countable transitive model of ZFC + GCH.
Let $\ka$ be any cardinal of $M$ of uncountable cofinality.
Suppose that $G$ is $\fin(\ka,2)$-generic over $M$, then 
in $M[G]$ the continuum is $\ka$ and for every
family $\ff$ of Borel subsets of $\N$ with size $\om_1<|\ff|<\ka$,
if $\bigcup \ff =\N$ then there exists $\ff_0\in [\ff]^{\om_1}$ with
$\bigcup\ff_0=\om^\om$.
\end{theorem}
This is only stated in \cite{AM2} for $\ka=\om_3$ but it is clear
from the proof that it is true in more generality.

\begin{corollary}\label{cohenreal}
Suppose $M$ is a countable transitive model of ZFC + GCH.
Let $\ka$ be any cardinal of $M$ of uncountable cofinality which is
not the successor of a cardinal of countable cofinality. 
Suppose that $G$ is $\fin(\ka,2)$-generic over $M$, then 
in $M[G]$ the continuum is $\ka$ and for every uncountable $\ga<\ka$
if $F:\ga^\om\to \om^\om$ is continuous and onto, then there
exists a $Q\in [\ga]^{\om_1}$ such that $F(Q^\om)=\om^\om$.
\end{corollary}

\begin{proof}
Let $\Sigma=[\ga]^\om\cap M$.  Note that $|\Sigma|<\ka$ 
since in $M$ $|\ga^\om|>\ga$ if and only if $\ga$ has cofinality $\om$,
but in that case $|\ga^\om|=\ga^+<\ka$.
Since the forcing is c.c.c. 
$$M[G]\models \;\;\ga^\om =\bigcup\{X^\om\;:\; X\in\Sigma\}$$
For any $X\in\Sigma$ the continuous image 
$F(X^\om)$ is a ${\bf\Sigma}^1_1$ set, and hence the union
of $\om_1$ Borel sets.  Given a family $\ff$ of $|\Sigma|$-many Borel sets
whose union is $\om^\om$ there is a subfamily $\ff_0$ of size $\om_1$
whose union is $\om^\om$ and hence a $Q$ as required.
\end{proof}

If $\ka=\lambda^+$ where $\lambda$ has cofinality $\om$, then
the result holds for $\ga<\lambda$ but since $\lambda^\om=\ka$ holds
in $M$ we do not know whether it is true for $\lambda$.

\begin{proposition}\label{prop:smallbutnotlarge}
It is consistent that the continuum be arbitrarily large,
$\N$ can be partitioned into $\w_2$ Borel sets, and 
$\N$ is not the condensation of $\ka^\om$ whenever
$\w_2 < \k < \continuum$ 
\end{proposition}
\begin{proof}
Let $M$ be a model of \textit{ZFC} satisfying Martin's Axiom and $\continuum = \w_2$. Using transfinite induction in $M$, it is possible to construct a sequence $\seq{C_\a}{\a < \w_2}$ of closed nowhere dense subsets of $\C$ such that for every non-atomic Borel probability measure $\mu$ on $\C$ there are countably many of the $C_\a$ whose union has $\mu$-measure $1$.

In $M$, force with the measure algebra on $2^\lambda$ for any $\lambda$ with uncountable cofinality not the successor of a cardinal of countable
cofinality. In the generic extension we have $\continuum = \lambda$, and every new real is random with respect to some non-atomic Borel probability measure in the ground model. Because of our choice of the $C_\a$, this implies that every new real will be in some $C_\a$. Therefore $\set{C_\a}{\a < \w_2} \cup \set{\{x\}}{x \in \C \setminus \bigcup_{\a < \w_2}C_\a}$ has size $\aleph_2$ in the generic extension. Because the $C_\a$ are disjoint, it is a partition of $\C$ into closed sets.

The proof of Theorem 3.7 in \cite{AM2} uses Cohen reals, but the same idea shows that this generic extension has the property that for every
family $\ff$ of Borel subsets of $\N$ with size $\aleph_2<|\ff|<\lambda$,
if $\bigcup \ff =\N$ then there exists $\ff_0\in [\ff]^{\om_2}$ with
$\bigcup\ff_0=\om^\om$.  As in the proof of Corollary \ref{cohenreal}
we get that $\N$ is not the condensation of any $\ka^\om$ whenever
$\aleph_2<\ka<\continuum$.
\end{proof}

Note the similarity of this argument to the argument of Stern in \cite{Stn} (later rediscovered by Kunen), where he proves that $\C$ can be partitioned into $\aleph_1$ closed sets in any random real extension of a model of \textit{CH}.

Note also that trivial modifications to the proof of Proposition~\ref{prop:smallbutnotlarge} allow us to replace $\w_2$ with any cardinal $\mu$ of uncountable cofinality. However, doing so will not guarantee partitions of all sizes smaller than $\mu$. This is because it is not currently known whether Martin's Axiom implies (or even permits) Borel set partitions of $\N$ of all sizes less than $\continuum$. Thus this proof does not give ``small partitions without large,'' but only ``partitions of a given size without larger ones.''

\begin{question}
If $\w_1 < \k < \continuum$, does Martin's Axiom imply the existence of a size-$\k$ partition of $\N$ into Borel sets? Is this consistent with Martin's Axiom?
\end{question}

\section{Similarity types}\label{sec:types}

This section deals mostly with perfect completely ultrametrizable spaces of size $\continuum$. For the sake of brevity, we will henceforth refer to such spaces as \textbf{PCU spaces}.

Say that two spaces are \textbf{similar} if each condenses onto the other. In \cite{Brn}, it is proved that there are exactly three similarity types of separable PCU spaces, that is, of perfect zero-dimensional Polish spaces. Furthermore, it is proved that these three types are totally ordered by the relation ``$X$ condenses onto $Y$'' (this is called the \textbf{condensation relation}). It is also proved in \cite{Brn} that if \textit{CH} is assumed then there are exactly four similarity types of PCU spaces, and these four types are totally ordered by the condensation relation. That is, the inclusion of non-separable spaces (equivalently, of uncountable trees) introduces only one new similarity type, namely those spaces representable with trees of size $\continuum$.

In general, neither
\begin{enumerate}
\item the total orderability of similarity types by condensation
\item the similarity of spaces represented by uncountable trees of the same size
\end{enumerate}
necessarily holds when \textit{CH} fails. If \textit{MA} holds, for example, then $\w_1^\w$ and $\w_1 \times \C$ are not similar, and neither of $\w_1 \times \C$ and $\N$ condenses onto the other (see \cite{Brn}, Proposition 5.10 and Corollary 6.3). It is left an open question in \cite{Brn} whether it is consistent for $(1)$ and $(2)$ to hold when \textit{CH} fails.

We will show in this section that it is consistent for $(1)$ and $(2)$ to hold with $\continuum$ equal to any $\w_n < \w_\w$. In fact the appropriate model has already been constructed in Theorem~\ref{thm:lotsofpartitions}, and here we merely need to show that $(1)$ and $(2)$ hold in this model.

\begin{lemma}\label{lem:CtoN}
The following are equivalent:
\begin{enumerate}
\item $\C$ can be partitioned into $\k$ closed sets.
\item $\C$ can be partitioned into $\k$ copies of $\C$
\item $\N$ can be partitioned into $\k$ copies of $\C$.
\end{enumerate}
\end{lemma}
\begin{proof}
If $\k$ is countable then the result is trivial, so suppose $\k \leq \continuum$ is uncountable.

$(3) \Rightarrow (1)$: Let $D$ be a countable dense subset of $\C$; then $\C \setminus D \homeo \N$. If $\set{K_\a}{\a \in \k}$ is a partition of $\C \setminus D$ into copies of $\C$, then $\set{K_\a}{\a \in \k} \cup \set{\{x\}}{x \in D}$ is a partition of $\C$ into $\k$ closed sets.

$(1) \Rightarrow (2)$: This is given by Proposition~\ref{prop:notjustpoints}.

$(2) \Rightarrow (3)$: Let $\set{K_\a}{\a \in \k}$ be a partition of $\C$ into $\k$ copies of $\C$. We say that a partition is ``nice'' if each partition element is nowhere dense. Assuming our partition is nice, we can pick a sequence $\seq{x_n}{n < \w}$ of points in $\C$ such that $\set{x_n}{n \in \w}$ is dense in $\C$ and no two points of our sequence are in the same $K_\a$. Then $X = \C \setminus \set{x_n}{n \in \w} \homeo \N$, and $\set{K_\a \cap X}{\a \in \k}$ is a partition of $X$. For each $\a$, either $K_\a$ contains no $x_n$, in which case $K_\a \cap X = K_\a \homeo \C$, or $K_\a \cap X = K_\a \setminus \{x_n\}$ for some $x_n \in K_\a$, in which case $K_\a \cap X \homeo \w \times \C$. This lets us obtain a partition of $X$ into $\k$ copies of $\C$.

Given any partition $\set{K_\a}{\a \in \k}$ of $\C$ into closed sets, we will now show how to find a nice partition. The basic idea is to do something like a Cantor-Bendixson derivative to eliminate partition elements with non-empty interior. Set $C_0 = \C$. Given $C_\a$, let
$$C_{\a+1} = C_\a \setminus \bigcup \set{U \sub C_\a}{U \text{ is clopen and, for some }\a, \, U \sub K_\a},$$
and if $\a$ is a limit ordinal take $C_\a = \bigcap_{\b < \a}C_\b$. By transfinite recursion, this defines a decreasing sequence $\seq{C_\a}{\a \in Ord}$ of closed subspaces of $C_0$. Because each $C_\a$ is closed and $\C$ is second countable, there is some countable ordinal $\a$ such that $C_\b = C_\a$ for all $\b \geq \a$. If $x$ were an isolated point of $C_\a$ then we would have $x \notin C_{\a+1}$, so $C_\a$ is perfect. By induction, again using the fact that $\C$ is second countable, there are for any $\b \leq \a$ at most countably many $\gamma$ such that $K_\gamma \cap C_\b = \0$. Since $\k$ is uncountable, $\set{K_\gamma}{\gamma \in \k}$ restricts to a partition of size $\k$ on $C_\a$. In particular, $C_\a \neq \0$; since we have already seen that $C_\a$ is closed in $\C$ and has no isolated points, $C_\a \homeo \C$. If $K_\gamma \cap C_\a$ had nonempty interior, there would be some clopen $U \sub C_\a$ with $U \sub K_\gamma$, contradicting the fact that $C_\a = C_{\a+1}$. Thus each $K_\gamma \cap C_\a$ is closed and nowhere dense in $C_\a$. As in the proof of Proposition~\ref{prop:notjustpoints}, $\set{(K_\gamma \cap C_\a) \times \C}{K_\gamma \cap C_\a \neq \0}$ is a partition of $C_\a \times \C \homeo \C$ into nowhere dense copies of $\C$.
\end{proof}

\begin{lemma}\label{lem:sizematters}
If $\card{T} < \card{S}$, then $\[S\]$ is not a condensation of $\[T\]$.
\end{lemma}
\begin{proof}
See \cite{Brn}, Proposition 5.5.
\end{proof}

\begin{theorem}\label{thm:nicetypes}
Assume $\continuum = \w_n < \w_\w$. The following are equivalent:
\begin{enumerate}
\item There are $n + 3$ similarity types of PCU spaces, and these are totally ordered by condensation.
\item There are $n + 3$ similarity types of PCU spaces.
\item The similarity types of PCU spaces are totally ordered by condensation.
\item $\C$ is a condensation of every PCU space.
\item For every $\k \leq \continuum$, there is a partition of $\C$ into $\k$ closed sets.
\end{enumerate}
\end{theorem}
\begin{proof}
$(1)$ implies $(2)$ and $(3)$ trivially. No Hausdorff space is a nontrivial condensation of $\C$, so $(3)$ implies $(4)$. For every $\k \leq \continuum$, $\k \times \C$ is a completely ultrametrizable space. If $f: \k \times \C \to \C$ is a condensation, then $\set{f(\{a\} \times \C)}{\a \in \k}$ is a partition of $\C$ into $\k$ copies of $\C$. Thus $(4)$ implies $(5)$.

Next we show that $(2)$ implies $(5)$. Suppose there are exactly $n+3$ similarity types of PCU spaces. By Lemma~\ref{lem:sizematters}, there are at least $n$ similarity types corresponding to PCU spaces arising from uncountable trees: one for each possible tree size. Because there are exactly three similarity types for PCU spaces arising from countable trees (see \cite{Brn}, Theorem 3.9), we have exactly $n$ types for uncountable trees, and $\[T\]$ and $\[S\]$ are similar if and only if $\card{S} = \card{T}$.

In particular, if $0 < m < n$ then $\w_m \times \C$ and $\w_m^\w$ are similar. Let $f: \w_m^\w \to \w_m \times \C$ be a condensation. By Lemma~\ref{lem:opensubsets}, $f^{-1}(\{0\} \times \C) \homeo \w_m^\w$. Thus $f \rest f^{-1}(\{0\} \times \C)$ is a condensation from (a homeomorphic copy of) $\w_m^\w$ to $\{0\} \times \C \homeo \C$. Composing with a condensation $\w_m \times \C \to \w_m^\w$, we see that there is a condensation $g: \w_m \times \C \to \C$. Since any Hausdorff condensation of $\C$ is simply $\C$, $\set{g(\{\a\} \times \C)}{\a \in \w_m}$ is a partition of $\C$ into $\aleph_m$ copies of $\C$. Thus $(2)$ implies $(5)$.

It remains to show that $(5)$ implies $(1)$. We will show that if $T$ is a tree with $\k$ nodes, with $\w < \k < \continuum$, then $\[T\]$ is similar to $\k^\w$. This shows that each uncountable $\k < \continuum$ corresponds to a single similarity type, consisting precisely of the spaces arising from trees of size $\k$ (and these types are distinct by Lemma~\ref{lem:sizematters}). Given our assumptions on $\continuum$ and the fact that there are exactly three types corresponding to countable trees (Theorem 3.9 from \cite{Brn}), we then have exactly $n+3$ types. We will also show that $\k^\w$ condenses onto $\lambda^\w$ whenever $\w \leq \lambda < \k \leq \continuum$, which will show that these types are totally ordered by condensation. It will be convenient to prove the latter of these propositions first.

\begin{claim}
If $\w \leq \lambda \leq \k \leq \continuum$, then $\lambda^\w$ is a condensation of $\k^\w$.
\end{claim}
\begin{proof}[Proof of claim]
Combining Corollary~\ref{cor:inductivedivisions} and Lemma~\ref{lem:CtoN} with $(5)$, we see that there is a partition of $\lambda^\w$ into $\k$ copies of $\C$. In particular, $\lambda^\w$ is a condensation of $\k \times \N$. By Theorem~\ref{thm:unexpectedconverse} and $(5)$, $\k \times \N$ is a condensation of $\k \times \k^\w \homeo \k^\w$. Composing condensations, we have a condensation from $\k^\w$ to $\lambda^\w$.
\end{proof}

Let $T$ be a tree and $\card{T} = \k$, with $\w < \k \leq \continuum$. We will show that $\[T\]$ is similar to $\k^\w$ by showing (in the following three claims) that there are condensations $\[T\] \to \k \times \C \to \k^\w \to \[T\]$.

\begin{claim}
$\k \times \C$ is a condensation of $\[T\]$.
\end{claim}
\begin{proof}[Proof of claim]
Because $\k$ is uncountable and regular, and because $T = \bigcup_{n \in \w}\lev{n}{T}$, there is some $n$ such that $\card{\lev{n}{T}} = \k$. $\[T\]$ is the disjoint union of $\set{\[T_s\]}{\l{s} = n}$, and each $\[T_s\]$ is clopen in $\[T\]$. Thus it suffices to show that $\C$ is a condensation of each $\[T_s\]$. This follows immediately from $(4)$, and $(4)$ follows from $(5)$ by Theorem 5.1 of \cite{Brn}.
\end{proof}

\begin{claim}
$\k^\w$ is a condensation of $\k \times \C$.
\end{claim}
\begin{proof}[Proof of claim]
It follows from Corollary~\ref{cor:inductivedivisions}, Lemma \ref{lem:CtoN}, and $(5)$ that $\k^\w$ can be partitioned into $\k$ copies of $\C$. This is equivalent to being a condensation of $\k \times \C$.
\end{proof}

\begin{claim}
$\[T\]$ is a condensation of $\k^\w$.
\end{claim}
\begin{proof}[Proof of claim]
We will prove this claim by induction. Our inductive hypothesis is that whenever $S$ is a tree of size $\lambda$ then there is a condensation $\lambda^\w \to \[S\]$. If $\lambda = \w$ then the inductive hypothesis becomes: every perfect zero-dimensional Polish space is a condensation of $\N$. This is a well-known classical result (see, e.g., Exercise 7.15 in \cite{Kec}). Assume now that the inductive hypothesis holds for every $\lambda < \k$.

By Theorem 4.3 in \cite{Brn}, we may assume that every node $s$ of $T$ has exactly $\card{T_s}$ children. We will build a $\k^{<\w}$-scheme $(B_s)_{s \in \k^{<\w}}$ in $\[T\]$ by recursion.

Set $B_\0 = \[T\]$. Assume now that $B_s$ has been defined and is equal to $\[T_t\]$ for some node $t \in T$. If $\card{T_t} < \k$, then by hypothesis there is a condensation $\card{T_t}^\w \to \[T_t\]$. Since we have already proved that there is a condensation $\k^\w \to \card{T_t}^\w$ (our first claim above), there is a condensation $g: \[\k^{<\w}_s\] \to \[T_t\]$. Define $B_r = g(\[\k^{<\w}_r\])$ for every extension $r$ of $s$. If $\card{T_t} = \k$, then $t$ has $\k$ children in $T$ by our choice of $T$. Enumerating these as $\set{t_\a}{\a < \k}$, we let $B_{s \cat \a} = \[T_{t_\a}\]$. This recursion defines a $\k^{<\w}$-scheme $(B_s)_{s \in \k^{<\w}}$.

Let $x \in \[\k^{<\w}\]$. If there is some $n$ such that $B_{x(n)} = \[T_t\]$ for some $t$ with fewer than $\k$ children, then $B_{x(m)}$ is defined by some embedding $g: B_{x(n)} \to X$ for all $m \geq n$. Because $g$ is an embedding, $\lim_{m \to \infty}\mathrm{diam}(B_{x(m)}) = 0$ and $\bigcap_{n \in \w}B_{x(n)} = g(x)$. Otherwise, $B_{x(n)}$ is always equal to some $\[T_t\]$, where $t$ has $\k$ children in $T$. Then (by an easy induction) $B_{x(n)} = \[T_{y(n)}\]$ for some $y \in \[T\]$ and every $n$. Since $\set{\[T_{y(n)}\]}{n < \w}$ is a local basis for $y$, $\lim_{n \to \infty}\mathrm{diam}(B_{x(n)}) = 0$ in this case too; also, clearly, $\bigcap_{n \in \w}B_{x(n)} = y$. Thus $(B_s)_{s \in \k^{<\w}}$ has vanishing diameter, and $\bigcap_{n \in \w}B_{x(n)} \neq \0$ for every $x \in \[\k^{<\w}\]$. Furthermore, it is clear from our construction that $B_s = \bigcup \set{B_t}{t \text{ is a child of } s}$ for every $s \in \k^{<\w}$. It follows from Lemma~\ref{lem:schemes} that the associated map of $(B_s)_{s \in \k^{<\w}}$ is a condensation.

This shows that the inductive hypothesis holds at $\k$ and completes the induction.
\end{proof}
This shows that uncountable trees of the same size represent similar spaces and completes the proof that $(5)$ implies $(1)$.
\end{proof}

\begin{corollary}
Let $n < \w$. It is consistent with \textit{ZFC} that $\continuum = \w_n$ and the five propositions listed in the statement of Theorem~\ref{thm:nicetypes} all hold.
\end{corollary}
\begin{proof}
Combine Theorem~\ref{thm:lotsofpartitions} with Theorem~\ref{thm:nicetypes}.
\end{proof}

This corollary provides partial answers to Questions 5.11 and 5.12 from \cite{Brn}. It leaves open the question of whether the similarity types can be totally ordered when $\continuum > \w_\w$.

It is worth pointing out that the proof of Theorem~\ref{thm:nicetypes} does not depend on $\continuum$ being small. If we consider a model guaranteed by Theorem~\ref{thm:lotsofpartitions} in which $\continuum > \w_\w$, then the proof of Theorem~\ref{thm:nicetypes} shows that two PCU spaces arising from trees of size $\k$ with $\aleph_0 < \k < \aleph_\w$ will be similar. This gives countably many similarity types of spaces with weight less than $\w_\w$, with the types totally ordered (with order type $\w$) by condensation.


\begin{thebibliography}{99}
\bibitem{Brn} W. R. Brian, ``Completely ultrametrizable spaces and continuous bijections,'' submitted for publication.
\bibitem{Hsd} F. Hausdorff, ``Summen von $\aleph_1$ Mengen,'' \emph{Fundamenta Mathematicae} \textbf{26} (1936), pp. 241-255.
\bibitem{Huh} B. Hughes, ``Trees and ultrametric spaces: a categorical equivalence,'' \emph{Advances in Mathematics} \textbf{189}, iss. 1 (2004), pp. 148-191.
\bibitem{Kec} A. S. Kechris, \textit{Classical Descriptive Set Theory}, Graduate Texts in Mathematics vol. 156, Springer-Verlag, 1995.
\bibitem{Kun} K. Kunen, \emph{Set Theory: An Introduction to Independence Proofs}. Studies in Logic and the Foundations of Mathematics, vol. 102. Elsevier, Amsterdam, The Netherlands, 1980.
\bibitem{AJL} A. J. Lemin, ``On ultrametrization of general metric spaces,'' \emph{Proceedings of the American Mathematical Society} \textbf{131}, no. 3 (2002), pp. 979-989.
\bibitem{AM1} A. W. Miller, ``Covering $2^\w$ with $\w_1$ disjoint closed sets,'' in \emph{The Kleene Symposium}, eds. J. Barwise, H. J. Keisler, and K. Kunen, North-Holland Publishing Company (1980), pp. 415-421.
\bibitem{AM2} A. W. Miller, ``Infinite combinatorics and definability,'' \emph{Annals of Pure and Applied Mathematical Logic} \textbf{41} (1989), pp. 179-203.
\bibitem{Stn} J. Stern, ``Partitions of the real line into $F_\s$ or $G_\delta$ subsets,'' \emph{C. R. Acad. Sci. Paris S\'er I} \textbf{284}, vol. 16 (1977), pp. 921-922.
\end{thebibliography}
\end{document}